\documentclass[a4paper]{amsart}
\usepackage{amsmath,amsfonts,amssymb}
\usepackage{graphicx}
\usepackage{hyperref}
\usepackage{natbib}

\newtheorem{theorem}{Theorem}
\newtheorem{corollary}[theorem]{Corollaire}
\newtheorem{definition}[theorem]{Definition}
\newtheorem{example}[theorem]{Exemple}
\newtheorem{lemma}[theorem]{Lemma}

\newtheorem{remark}[theorem]{Remarque}

\begin{document}
\title{On the Enumeration of Vuza Canons}
\author{Franck Jedrzejewski}
\maketitle

\begin{abstract}
A Vuza canon of $\mathbb{Z}_{N}$ is a non periodic factorization of the
cyclic group $\mathbb{Z}_{N}$ in two factors. The aim of this paper is to
present some new results for generating Vuza canons, and to give new
minoration of their number for some values of N.

\medskip



\end{abstract}



\bigskip

\section{Rhythmic Canons}

Let $G$ be a finite cyclic group, and $S,R$\ two non-empty subsets of $G$.\
If each element $g\in G$ can be expressed uniquely in the form $g=s+r$ with $%
s\in S$ and $r\in R$ then the equation $G=S\oplus R$ is called a \emph{%
factorization} of $G$. A \emph{rhythmic canon} of $G$ is a factorization of $%
G$ in two subsets.\ The set $S$\ is called the \emph{inner voice}, and the
set $R$ is the \emph{outer voice}.\ If $G$ is the cyclic group $\mathbb{Z}%
_{N}=\{0,1,2,...,N-1\}$, it is worthy of noting that the set $S$ can be seen
as a tile of the line $\mathbb{Z}_{N}$ and we said that $S$ tiles $\mathbb{Z}%
_{N}$ with\ $R$ or that $R$ tiles with $S$, since if $(S,R)$ is a rhythmic
canon, also is $(R,S)$.

A non-empty set $S$ is \emph{periodic} if there is an element $g$ of $G$
such that $S+g=S$ and $g\neq 0$ (the identity element of $G$). Around 1948,
Haj\'{o}s thought that in a factorization of $\mathbb{Z}_{N}$ in two
factors, one of the factors had to be periodic. In fact, this conjecture is
false, and counterexamples are precisely Vuza canons. They were introduced
by D.T. Vuza under the name of \textquotedblleft Regular Complementary
Canons of Maximal Category\textquotedblright\ in his seminal paper %
\citep{Vuz1991} and are now known as \textquotedblleft Vuza
Canons\textquotedblright .

\begin{definition}
Let $G$ be a finite cyclic group, S,R two non-empty subsets of $G$. A Vuza
Canon $(S,R)$ of $G$ is a rhythmic canon 
\begin{equation*}
G=S\oplus R
\end{equation*}%
where neither S, nor R is periodic.
\end{definition}

A factorization of a cyclic group $G$ is said to be trivial if at least one
of its factors is periodic.\ If every factorization of $G$ is trivial, $G$
is called a \textquotedblleft good group\textquotedblright .\ Otherwise, if
non-trivial factorizations exist, $G$ is called a \textquotedblleft bad
group\textquotedblright .\ In 1950, G.\ Haj\'{o}s \citep{Haj1950b} proved
that there exist cyclic groups admitting a non-trivial factorization.\ N.G.\
de Bruijn \citep{Bru1953} showed that if $N=d_{1}d_{2}d_{3}$ where $d_{3}>1$%
, $d_{1}$ and $d_{2}$ are coprime and if both $d_{1}$ and $d_{2}$ are
composite numbers then the cyclic group $\mathbb{Z}_{N}$ of order $N$ is
bad. In \citep{Bru1953b}, N.G.\ de Bruijn remarked that the cyclic groups
which are not covered by this result are those of orders $p^{\alpha }$ $%
(\alpha \geq 1),$ $p^{\alpha }q$ $(\alpha \geq 1),$ $p^{2}q^{2},$ $p^{\alpha
}qr$ $(\alpha =1,2),$ $pqrs$ where $p,q,r,s$ denote different primes. Red%
\'{e}i \citep{Red1950} proved that the cyclic groups of orders $p^{\alpha }$ 
$(\alpha \geq 1),$ $pq,$ $pqr$ are \textquotedblleft good\textquotedblright
. N.G.\ de Bruijn \citep{Bru1953b} showed that $p^{\alpha }q$ $(\alpha \geq
1)$\ are \textquotedblleft good groups\textquotedblright . The remaining
cases $p^{2}q^{2},$ $p^{2}qr$ and $pqrs$ were showed by\ A.D.\ Sands %
\citep{San1957}\ in 1957. Later on, the classification was extended to all
finite abelian groups \citep{San1962, San2004}. Vuza gave the first musical
application of the factorization of cyclic groups and computed the first
\textquotedblleft Vuza canon\textquotedblright\ that appears for $N=72$, and
the next one occurs at $N=108$, 120, 144, 168, etc. He also showed the
following equivalence of the conditions on order $N$ \citep{Vuz1991}.

\begin{theorem}
The following statements are equivalent:

\noindent $(i)$ Vuza canons only exist for orders N which are NOT of the form%
\begin{equation*}
p^{\alpha }~(\alpha \geq 1),~p^{\alpha }q~(\alpha \geq
1),~p^{2}q^{2},~p^{\alpha }qr~(\alpha =1,2),~pqrs
\end{equation*}%
where p, q, r, s are different primes.

\noindent $(ii)$ Vuza canons only exist for orders N of the form%
\begin{equation*}
N=n_{1}n_{2}n_{3}p_{1}p_{2}
\end{equation*}%
where $p_{1},p_{2}$\ denote different primes, $p_{i}n_{i}\geq 2,$ for $i=1,2$
and $n_{1}p_{1}$ and $n_{2}p_{2}$ are coprime. Such $N$ are called \emph{%
Vuza orders}.
\end{theorem}

Let $\mathbb{Z}_{b}$ be the cyclic group $\mathbb{Z}_{b}=\{0,1,2,...,b-1\}.$
The subset $\mathbb{I}_{a}$\ of $\mathbb{Z}_{b}$ is defined by%
\begin{equation*}
\mathbb{I}_{a}=\{0,1,2,...,a-1\}
\end{equation*}%
and if $k$ is a positive integer, $k\mathbb{I}_{a}$ denotes the subset of $%
\mathbb{Z}_{b}$ 
\begin{equation*}
k\mathbb{I}_{a}=\{0,k~\mathop{\rm mod}\nolimits~b,2k~\mathop{\rm mod}%
\nolimits~b,...,(a-1)k~\mathop{\rm mod}\nolimits~b\}
\end{equation*}%
As shown by E.\ Amiot, rhythmic canons are stable by duality \citep{Ami2011}%
, which is the exchange between inner and outer voices, and by affine
transformation \citep{Ami2005} $x\rightarrow ax+b$, where $a$ is coprime
with $N$. Vuza canon are also stable by duality.\ If $(S,R)$ is a Vuza
canon, then $(R,S)$ is also a Vuza canon as the addition is commutative. The
stability by affine transformation is not obvious. However, the stability by
concatenation of rhythmic canons states as follows.

\begin{theorem}
Let k be a positive integer.\ If $(S,R)$ is a rhythmic canon of $\mathbb{Z}%
_{N}$, then $(S\oplus N\mathbb{I}_{k},R)$ is a rhythmic canon of $\mathbb{Z}%
_{kN}$.
\end{theorem}

\begin{proof}
Since $(S,R)$ is a rhythmic canon of $\mathbb{Z}_{N}$, we have $S\oplus N%
\mathbb{I}_{k}\oplus R=\mathbb{I}_{N}\oplus N\mathbb{I}_{k}=\mathbb{Z}_{kN}.$
\end{proof}

Remark that $(S\oplus N\mathbb{I}_{k},R)$\ is not a Vuza canon, since $%
S\oplus N\mathbb{I}_{k}$\ is $N$ periodic. However in some cases, the
converse is true for some Vuza canons.\ The following lemma is shown in %
\citep{Sza2009}.

\bigskip

\begin{lemma}
Let G be a finite cyclic group, H a subgroup of G and $S,R$ two subsets of G
such that $S\subseteq H$ then%
\begin{equation*}
(S+R)\cap H=S+(R\cap H)
\end{equation*}
\end{lemma}

\begin{proof}
1) Let $g$ be an element of $S+(R\cap H)$ that is $g=s+r$ with $s\in S$ and $%
r\in R\cap H.$ Since $g$ is also element of $S+R$\ and element of $H,$ we
have the inclusion%
\begin{equation*}
S+(R\cap H)\subseteq (S+R)\cap H
\end{equation*}%
2) Conversely, given $g\in (S+R)\cap H,$ $g$ belongs to $S+R$ and $g\in H.$
Since $g\in S+R,$ there exist $s\in S$ and $r\in R$ such that $g=s+r$.\ But $%
g$ is in $H$ and $s\in S\subseteq H.\ $Thus $r\in H$ as $H$ is a subgroup of 
$G$. Thus, the element $r$ is in $R\cap H$ and so%
\begin{equation*}
(S+R)\cap H\subseteq S+(R\cap H)
\end{equation*}
\end{proof}

\begin{theorem}
Let G be a finite cyclic group and H a subgroup of G.\ If $(S,R)$ is a Vuza
canon of $G$ and if $S$ is a subset of $H$, $(R\cap H)$ is non periodic in $%
H $, then $(S,R\cap H)$ is a Vuza canon of $H$.
\end{theorem}

\begin{proof}
The equations%
\begin{eqnarray*}
H &=&G\cap H \\
&=&(S+R)\cap H \\
&=&S+(R\cap H)
\end{eqnarray*}%
show that $H$ is the sum of $S$ and $R\cap H$, as $S$ is a subset of $H$.
This sum is direct since the sum $S+R$\ is. Moreover, suppose that $S$ is
periodic in $H$. Since $S\subset H\subset G,$ the set $S$ is periodic in $G$%
.\ And this leads to a contradiction
\end{proof}

\begin{corollary}
Let k be a positive integer and N a Vuza order.\ If $(S,R)$ is a Vuza canon
of $\mathbb{Z}_{kN}$, $S$ and $(R\cap k\mathbb{I}_{N})$\ two non-periodic
subsets of $k\mathbb{I}_{N},$ then the pair $(S^{\prime },R^{\prime })$
defined by 
\begin{equation*}
S^{\prime }=S/k,\text{ }R^{\prime }=(R\cap k\mathbb{I}_{N})/k
\end{equation*}%
is a Vuza canon of $\mathbb{Z}_{N}.$
\end{corollary}

\begin{proof}
Since $k\mathbb{I}_{N}$ can be seen as a subgroup of $\mathbb{Z}_{kN},$ we
can apply the previous theorem. These shows that $(S,R\cap k\mathbb{I}_{N})$
is a Vuza canon of $k\mathbb{Z}_{N}.~$ Taking the quotient by $k$, since $%
S^{\prime },R^{\prime }$ are in $\mathbb{Z}_{N},$ $(S^{\prime },R^{\prime })$
is a Vuza canon of $\mathbb{Z}_{N}.$
\end{proof}

\begin{example}
For instance, consider the Vuza canon $S=\{0,16,32,36,52,68\}$ and%
\begin{eqnarray*}
R &=&\{0,7,12,15,24,33,34,45,46,55,57,58,63,72,84,96,103,105,106, \\
&&111,117,118,129,130\}
\end{eqnarray*}%
of $\mathbb{Z}_{Nk}=\mathbb{Z}_{144}$ with $N=72$ and $k=2$. Since $H=2%
\mathbb{Z}_{72}$ is a subgroup of $\mathbb{Z}_{144},$ the pair 
\begin{eqnarray*}
S^{\prime } &=&S/2=\{0,8,16,18,26,34\} \\
R^{\prime } &=&(R\cap 2\mathbb{I}_{72})/2=%
\{0,6,12,17,23,29,36,42,48,53,59,65\}
\end{eqnarray*}%
is a Vuza canon of $\mathbb{Z}_{72}.$
\end{example}

\section{Constructing Vuza Canons}

In the following, we need two results.\ The first gives us a criteria to
show that a sum of two finite sets is a direct sum. And the second result
helps us to factorize cyclic groups.\ 

\begin{theorem}
Let G be a finite cyclic group and let S,R be two non-empty subsets of G.\
The following statements are equivalent.

(i) The sum $S+R$\ is direct and is equal to $G$.

(ii) $G=S+R$ and $\left\vert G\right\vert =\left\vert S\right\vert
\left\vert R\right\vert $.
\end{theorem}

The proof of this theorem is given in many textbooks on group theory.\ It is
also in the book of Szabo and Sands \citep{Sza2009}.

\begin{lemma}
\label{lemma_22}Let $a,b,c$ be positive integers such that $a$ is coprime
with $c$.\ We have%
\begin{equation*}
\begin{tabular}{lll}
$(i)$ & \quad & $\mathbb{I}_{a}\oplus a\mathbb{I}_{b}=\mathbb{Z}_{ab}$ \\ 
$(ii)$ &  & $c\mathbb{I}_{a}\oplus a\mathbb{I}_{bc}=\mathbb{Z}_{abc}$%
\end{tabular}%
\end{equation*}
\end{lemma}

\begin{proof}
(i) Writing the sum as a table where each line $j$ corresponds to the set $%
ja+\mathbb{I}_{a}=\{ja,ja+1,ja+2,...,ja+a-1\}$ for $j=0,1,...,b-1$ and
reading the table line by line leads to the set $\mathbb{Z}_{ab}.$
Futhermore, the sum is direct since $\left\vert \mathbb{I}_{a}\right\vert
\left\vert a\mathbb{I}_{b}\right\vert =ab=\left\vert \mathbb{Z}%
_{ab}\right\vert .$

(ii) To prove the second statement, we write the table corresponding to the
sum line by line. Starting with $a\mathbb{I}_{bc}$ in the first line, we add
for each line the term $cj$.\ The sum is computed mod $abc$. As above, write
the sum as a table and then rewrite each line in lexicographic order. Then
each column $k$ is the set $k+\mathbb{I}_{a}$ with $k=0,a,2a,...,(bc-1)a.$
Taking the union column by column leads to%
\begin{eqnarray*}
&&\mathbb{I}_{a}\cup (a+\mathbb{I}_{a})\cup (2a+\mathbb{I}_{a})\cup \cdots
\cup ((bc-1)a+\mathbb{I}_{a}) \\
&=&\mathbb{I}_{a}+a\mathbb{I}_{bc}=\mathbb{Z}_{abc}
\end{eqnarray*}%
by the first property (i) and because since $a$ is coprime with $c$, the
collection of sets $ja+\mathbb{I}_{a}$ with $j=0,1,...,b-1$ is pairwise
disjoint.\ The sum is direct since $\left\vert c\mathbb{I}_{a}\right\vert
\left\vert a\mathbb{I}_{bc}\right\vert =a\times bc=\left\vert \mathbb{Z}%
_{abc}\right\vert .$
\end{proof}

\begin{example}
For $a=5,b=7$ and $c=2$, the sum is presented in the following table
(computations are done modulo $abc=70$)%
\begin{equation*}
\begin{tabular}{|l|l|l|l|l|l|l|}
\hline
$0$ & $5$ & $10$ & $15$ & $\cdots $ & $60$ & $65$ \\ \hline
$2$ & $7$ & $12$ & $17$ & $\cdots $ & $62$ & $67$ \\ \hline
$4$ & $9$ & $14$ & $19$ & $\cdots $ & $64$ & $69$ \\ \hline
$6$ & $11$ & $16$ & $21$ & $\cdots $ & $66$ & $1$ \\ \hline
$8$ & $13$ & $18$ & $23$ & $\cdots $ & $68$ & $3$ \\ \hline
\end{tabular}%
\end{equation*}%
Rewriting the last two lines in lexicographic order shows that the first
column becomes the set $\mathbb{I}_{5}=\{0,1,2,3,4\}$ and the next one are $%
k+\mathbb{I}_{a}.$ The union of all columns leads to $\mathbb{Z}_{70}$.
\end{example}

\begin{theorem}
Let $k$ be a positive integer such that $kN$ is a Vuza order.\ If the pair $%
(S,R)$ is a Vuza canon of $\mathbb{Z}_{N}$, then $(S^{\prime },R^{\prime })$
with%
\begin{eqnarray*}
S^{\prime } &=&kS\oplus \mathbb{I}_{k} \\
R^{\prime } &=&kR
\end{eqnarray*}%
is a Vuza canon of $\mathbb{Z}_{kN}$.
\end{theorem}

\begin{proof}
By the lemma \ref{lemma_22}, we have%
\begin{equation*}
S^{\prime }\oplus R^{\prime }=k(S\oplus R)\oplus \mathbb{I}_{k}=\mathbb{I}%
_{k}\oplus k\mathbb{I}_{N}=\mathbb{Z}_{kN}
\end{equation*}%
$S^{\prime }$ and $R^{\prime }$ are non-periodic in $\mathbb{Z}_{kN}$,
otherwise $S$ and $R\mathbb{\ }$will be periodic in $\mathbb{Z}_{N}.$Then $%
(S^{\prime },R^{\prime })$ is a Vuza canon of $\mathbb{Z}_{kN}.$
\end{proof}

In 2004 at the Mamux Seminar, F.\ Jedrzejewski gave a solution for
constructing large Vuza canons, depending only on the parameters $%
n_{1},n_{2},n_{3},p_{1}$ and $p_{2}$ (see \citep{Jed2006, Jed2009}). Let us
rephrase this result in the following manner:

\begin{theorem}
\label{thm_11}Let $p_{1}$, $p_{2}$ be different prime numbers, and $%
n_{1},n_{2},n_{3}$ be positive integers such that the product $n_{1}p_{1}$
is coprime with $n_{2}p_{2}$, then the pair $(S,R)$ defined by%
\begin{eqnarray*}
S &=&A+B \\
R &=&(U+V^{\prime }+K_{1})\cup (U^{\prime }+V+K_{2})
\end{eqnarray*}%
is a Vuza canon of $\mathbb{Z}_{N}$ with $N=n_{1}n_{2}n_{3}p_{1}p_{2}$ if $R$
is a non-periodic subset of $\mathbb{Z}_{N}$\ and:%
\begin{equation*}
\begin{tabular}{lll}
$A=n_{1}p_{1}n_{3}\mathbb{I}_{n_{2}}$ & $\qquad $ & $B=n_{2}p_{2}n_{3}%
\mathbb{I}_{n_{1}}$ \\ 
$U=n_{1}n_{2}n_{3}p_{1}\mathbb{I}_{p_{2}}$ & $\qquad $ & $%
V=n_{1}n_{2}n_{3}p_{2}\mathbb{I}_{p_{1}}$ \\ 
$U^{\prime }=n_{2}n_{3}\mathbb{I}_{p_{2}}$ & $\qquad $ & $V^{\prime
}=n_{1}n_{3}\mathbb{I}_{p_{1}}$ \\ 
$K_{1}=\{0\}$ & $\qquad $ & $K_{2}=\{1,2,...,n_{3}-1\}$%
\end{tabular}%
\end{equation*}
\end{theorem}

The proof of this theorem is straightforward using the lemma \ref{lemma_22}.
This theorem has at least two generalizations.\ The first one consists in
choosing the sets $K_{1}$ and $K_{2}$ amongst a complete set $K$\ of coset
representatives for $\mathbb{Z}_{N}$ modulo a proper non-zero subgroup $H$.\
The non-periodicity of the subset $S$ is given by the following theorem
which is proved in \cite{Sza2009}. A subset $A$ of a group $G$ is \emph{%
normalized} if it contains the unit element of $G$ (0 in additive notation).
The smallest subgroup of $G$ that contains $A$ is denoted by $\left\langle
A\right\rangle $.

\begin{theorem}
\label{thm_szabo}If $A,B$ are normalized non-periodic sets of a cyclic group
and the sum $A+\left\langle B\right\rangle $ is direct, then $A+B$ is a
non-periodic set.
\end{theorem}

\begin{theorem}
\label{thm_12}Let $p_{1}$, $p_{2}$ be different prime numbers, $%
n_{1},n_{2},n_{3}$ be positive integers such that $n_{1}p_{1}$ is coprime
with $n_{2}p_{2}$, and $H$ be the subgroup $H=n_{3}\mathbb{I}%
_{n_{1}p_{1}n_{2}p_{2}}$ of $\mathbb{Z}_{N}$ with $%
N=n_{1}n_{2}n_{3}p_{1}p_{2},$\ and let $K$ be a complete set of cosets
representatives for $\mathbb{Z}_{N}$ modulo H such that $K$ is the disjoint
union $K=K_{1}\cup K_{2}$, then the pair $(S,R)$ defined by%
\begin{eqnarray*}
S &=&A+B \\
R &=&(U+V^{\prime }+K_{1})\cup (U^{\prime }+V+K_{2})
\end{eqnarray*}%
is a Vuza canon of $\mathbb{Z}_{N}$ if $R$ is a non-periodic subset of $%
\mathbb{Z}_{N}$\ and:%
\begin{equation*}
\begin{tabular}{lll}
$A=n_{1}p_{1}n_{3}\mathbb{I}_{n_{2}}$ & $\qquad $ & $B=n_{2}p_{2}n_{3}%
\mathbb{I}_{n_{1}}$ \\ 
$U=n_{1}n_{2}n_{3}p_{1}\mathbb{I}_{p_{2}}$ & $\qquad $ & $%
V=n_{1}n_{2}n_{3}p_{2}\mathbb{I}_{p_{1}}$ \\ 
$U^{\prime }=\alpha n_{2}n_{3}\mathbb{I}_{p_{2}}$ & $\qquad $ & $V^{\prime
}=\beta n_{1}n_{3}\mathbb{I}_{p_{1}}$%
\end{tabular}%
\end{equation*}%
and $\alpha =1,n_{1}$ or $p_{1}$, and $\beta =1,n_{2}$ or $p_{2}.$
\end{theorem}

\begin{proof}
In order to calculate the sum $S+R$, suppose that $A+U+V^{\prime }=A+U+V$
and $B+V+U^{\prime }=B+V+U,$ then the sum is given by 
\begin{eqnarray*}
S+R &=&\left( A+B\right) +\left( (U+V^{\prime }+K_{1})\cup (U^{\prime
}+V+K_{2})\right) \\
&=&\left( A+B+U+V^{\prime }+K_{1}\right) \cup \left( A+B+U^{\prime
}+V+K_{2}\right) \\
&=&\left( A+B+U+V+K_{1}\right) \cup \left( A+B+U+V+K_{2}\right) \\
&=&A+B+U+V+\left( K_{1}\cup K_{2}\right) \\
&=&\left( A+U\right) +\left( B+V\right) +K
\end{eqnarray*}%
By the lemma \ref{lemma_22},%
\begin{equation*}
\begin{tabular}{l}
$X=A+U=n_{1}n_{3}p_{1}\mathbb{I}_{n_{2}p_{2}}$ \\ 
$Y=B+V=n_{2}n_{3}p_{2}\mathbb{I}_{n_{1}p_{1}}$%
\end{tabular}%
\end{equation*}%
it follows that 
\begin{equation*}
H=X+Y=n_{3}\mathbb{I}_{n_{1}p_{1}n_{2}p_{2}}
\end{equation*}%
and the sum is equal to the cyclic group:%
\begin{eqnarray*}
S+R &=&X+Y+K \\
&=&H+K=\mathbb{Z}_{N}
\end{eqnarray*}%
Since $\left\vert S\right\vert \left\vert R\right\vert =\left\vert
H\right\vert \left\vert K\right\vert =n_{1}p_{1}n_{2}p_{2}\times n_{3}=N,$
the sum $S+R$ is direct. To prove that $S$ is non-periodic, we apply the
theorem \ref{thm_szabo}.\ The smallest subgroup that contains $B$ is $%
\left\langle B\right\rangle =n_{2}p_{2}n_{3}\mathbb{I}_{n_{1}p_{1}}$ and the
sum 
\begin{equation*}
A+\left\langle B\right\rangle =n_{3}(n_{1}p_{1}\mathbb{I}_{n_{2}}+n_{2}p_{2}%
\mathbb{I}_{n_{1}p_{1}})
\end{equation*}%
is direct since $n_{1}p_{1}$ is coprime with $n_{2}p_{2}$.\ Suppose that
there are $x,x^{\prime }\in A$ and $y,y^{\prime }\in \left\langle
B\right\rangle $ verifying the equation%
\begin{equation*}
n_{1}p_{1}x+n_{2}p_{2}y=n_{1}p_{1}x^{\prime }+n_{2}p_{2}y^{\prime }
\end{equation*}%
The solutions are $x-x^{\prime }=-n_{2}p_{2}k$ and $y^{\prime
}-y=n_{1}p_{1}k $ where $k$ is an integer.\ Since $x$ and $x^{\prime }$ are
non-negative elements of $A,$\ multiple of $n_{1}p_{1}$ and $n_{1}p_{1}$ is
coprime with $n_{2}p_{2},$ it follows that $k=0,$ $x=x^{\prime }$ and $%
y=y^{\prime }$. Thus, there is only one decomposition of each element of $%
A+\left\langle B\right\rangle .\ $The sum $A+\left\langle B\right\rangle $
is direct and by the theorem above $S$ is a non-periodic subset of $\mathbb{Z%
}_{N}.$ If $R$ is non-periodic, it follows that $(S,R)$ is a Vuza canon of $%
\mathbb{Z}_{N}$.\ Now, it remains to show that%
\begin{eqnarray*}
A+U+V^{\prime } &=&A+U+V \\
X+V &=&X+V^{\prime }
\end{eqnarray*}%
Since $\beta =1$ or $\beta $ is coprime with $p_{1}$, we have by the lemma %
\ref{lemma_22}: 
\begin{eqnarray*}
X+V^{\prime } &=&n_{1}p_{1}n_{3}\mathbb{I}_{n_{2}p_{2}}+\beta n_{1}n_{3}%
\mathbb{I}_{p_{1}} \\
&=&n_{1}n_{3}(\beta \mathbb{I}_{p_{1}}+p_{1}\mathbb{I}_{n_{2}p_{2}}) \\
&=&n_{1}n_{3}\mathbb{I}_{n_{2}p_{1}p_{2}}
\end{eqnarray*}%
In the same way,%
\begin{eqnarray*}
X+V &=&n_{1}p_{1}n_{3}\mathbb{I}_{n_{2}p_{2}}+n_{1}n_{2}n_{3}p_{2}\mathbb{I}%
_{p_{1}} \\
&=&n_{1}n_{3}(p_{1}\mathbb{I}_{n_{2}p_{2}}+n_{2}p_{2}\mathbb{I}_{p_{1}}) \\
&=&n_{1}n_{3}\mathbb{I}_{n_{2}p_{1}p_{2}}
\end{eqnarray*}%
Thus $X+V^{\prime }=X+V$. The same expression can be set for $B$:%
\begin{eqnarray*}
Y+U^{\prime } &=&n_{2}p_{2}n_{3}\mathbb{I}_{n_{1}p_{1}}+\alpha n_{2}n_{3}%
\mathbb{I}_{p_{2}} \\
&=&n_{2}n_{3}(p_{2}\mathbb{I}_{n_{1}p_{1}}+\alpha \mathbb{I}_{p_{2}}) \\
&=&n_{2}n_{3}\mathbb{I}_{n_{1}p_{1}p_{2}}
\end{eqnarray*}%
and%
\begin{eqnarray*}
Y+U &=&n_{2}p_{2}n_{3}\mathbb{I}_{n_{1}p_{1}}+n_{1}n_{2}n_{3}p_{1}\mathbb{I}%
_{p_{2}} \\
&=&n_{2}n_{3}(p_{2}\mathbb{I}_{n_{1}p_{1}}+n_{1}p_{1}\mathbb{I}_{p_{2}}) \\
&=&n_{2}n_{3}\mathbb{I}_{n_{1}p_{1}p_{2}}
\end{eqnarray*}%
and consequently,%
\begin{equation*}
Y+U^{\prime }=Y+U
\end{equation*}
\end{proof}

The demonstration depends only on the formulas of the lemma.\ Changing the
definition of $H$ leads to a new theorem.\ The proof keeps the same
arguments by exchanging $p_{1}$ and $n_{3}$ with $\alpha =n_{1}p_{1},\beta
=n_{2}p_{2}.$

\begin{theorem}
Let $p_{1}$, $p_{2}$ be different prime numbers, and $n_{1},n_{2},n_{3}$ be
positive integers such that $n_{1}p_{1}$ is coprime with $n_{2}p_{2}$, and $%
n_{1}n_{3}$ is coprime with $n_{2}p_{2}.$\ Let $H$ be the subgroup $H=p_{1}%
\mathbb{I}_{n_{1}n_{2}n_{3}p_{2}}$ of $\mathbb{Z}_{N}$ with $%
N=n_{1}n_{2}n_{3}p_{1}p_{2},$ and $K$ be a complete set of cosets
representatives for $\mathbb{Z}_{N}$ modulo H such that $K=K_{1}\cup K_{2}$,
then the pair $(S,R)$ defined by%
\begin{eqnarray*}
S &=&A+B \\
R &=&(U+V^{\prime }+K_{1})\cup (U^{\prime }+V+K_{2})
\end{eqnarray*}%
is a Vuza canon of $\mathbb{Z}_{N}$ if $R$ is a non-periodic subset of $%
\mathbb{Z}_{N}$\ and:%
\begin{equation*}
\begin{tabular}{lll}
$A=n_{1}p_{1}n_{3}\mathbb{I}_{n_{2}}$ & $\qquad $ & $B=n_{2}p_{2}p_{1}%
\mathbb{I}_{n_{1}}$ \\ 
$U=n_{1}n_{2}n_{3}p_{1}\mathbb{I}_{p_{2}}$ & $\qquad $ & $V=n_{1}p_{1}%
\mathbb{I}_{n_{3}}$ \\ 
$U^{\prime }=n_{2}p_{1}\mathbb{I}_{p_{2}}$ & $\qquad $ & $V^{\prime
}=n_{1}p_{1}n_{2}p_{2}\mathbb{I}_{n_{3}}$%
\end{tabular}%
\end{equation*}
\end{theorem}

\noindent Or to its dual version:

\begin{theorem}
Let $p_{1}$, $p_{2}$ be different prime numbers, and $n_{1},n_{2},n_{3}$ be
positive integers such that $n_{1}p_{1}$ is coprime with $n_{2}p_{2}$, and $%
n_{1}p_{1}$ is coprime with $n_{2}n_{3}.$ Let $H$ be the subgroup $H=p_{2}%
\mathbb{I}_{n_{1}n_{2}n_{3}p_{1}}$ of $\mathbb{Z}_{N}$ with $%
N=n_{1}n_{2}n_{3}p_{1}p_{2}$, and $K$ be a complete set of cosets
representatives for $\mathbb{Z}_{N}$ modulo $H$ such that $K=K_{1}\cup K_{2}$%
, then the pair $(S,R)$ defined by%
\begin{eqnarray*}
S &=&A+B \\
R &=&(U+V^{\prime }+K_{1})\cup (U^{\prime }+V+K_{2})
\end{eqnarray*}%
is a Vuza canon of $\mathbb{Z}_{N}$ if $R$ is a non-periodic subset of $%
\mathbb{Z}_{N}$\ and:%
\begin{equation*}
\begin{tabular}{lll}
$A=n_{1}p_{1}p_{2}\mathbb{I}_{n_{2}}$ & $\qquad $ & $B=n_{2}n_{3}p_{2}%
\mathbb{I}_{n_{1}}$ \\ 
$U=n_{1}n_{2}p_{1}p_{2}\mathbb{I}_{n_{3}}$ & $\qquad $ & $%
V=n_{1}n_{2}n_{3}p_{2}\mathbb{I}_{p_{1}}$ \\ 
$U^{\prime }=n_{2}p_{2}\mathbb{I}_{n_{3}}$ & $\qquad $ & $V^{\prime
}=n_{1}p_{2}\mathbb{I}_{p_{1}}$%
\end{tabular}%
\end{equation*}
\end{theorem}

The second generalization of the theorem \ref{thm_11} consists in
constructing the sets $U^{\prime }$ and $V^{\prime }$ by adding a set of
elements of respectively $B$ and $A$.

\begin{theorem}
Let $p_{1}$, $p_{2}$ be different prime numbers, $n_{1},n_{2},n_{3}$ be
positive integers such that $n_{1}p_{1}$ is coprime with $n_{2}p_{2}$, and $%
H $ be the subgroup $H=n_{3}\mathbb{I}_{n_{1}p_{1}n_{2}p_{2}}$ of $\mathbb{Z}%
_{N}$ with $N=n_{1}n_{2}n_{3}p_{1}p_{2},$\ and let $K$ be a complete set of
cosets representatives for $\mathbb{Z}_{N}$ modulo H such that $K$ is the
disjoint union $K=K_{1}\cup K_{2}$.\ Let A, B, U and V the sets%
\begin{equation*}
\begin{tabular}{lll}
$A=n_{1}p_{1}n_{3}\mathbb{I}_{n_{2}}$ & $\qquad $ & $B=n_{2}p_{2}n_{3}%
\mathbb{I}_{n_{1}}$ \\ 
$U=n_{1}n_{2}n_{3}p_{1}\mathbb{I}_{p_{2}}$ & $\qquad $ & $%
V=n_{1}n_{2}n_{3}p_{2}\mathbb{I}_{p_{1}}$%
\end{tabular}%
\end{equation*}%
Construct the set $U^{\prime }$ by replacing non-zero elements $u_{j}\in U$
by their sum with some non-zero elements of B (non necessary different) 
\begin{equation*}
U^{\prime }=\left( U\backslash \{u_{1},...,u_{n}\}\right) \cup
\{u_{1}+b_{1},...,u_{n}+b_{n}\}
\end{equation*}%
and the set $V^{\prime }$ by replacing some non-zero elements $v_{i}\in V$
by their sum with some non-zero elements of A%
\begin{equation*}
V^{\prime }=\left( V\backslash \{v_{1},...,v_{m}\}\right) \cup
\{v_{1}+a_{1},...,v_{m}+a_{m}\}
\end{equation*}%
Then the pair 
\begin{eqnarray*}
S &=&A+B \\
R &=&(U+V^{\prime }+K_{1})\cup (U^{\prime }+V+K_{2})
\end{eqnarray*}%
is a Vuza canon of $\mathbb{Z}_{N}$ if $R$ is non-periodic.
\end{theorem}

\begin{proof}
Since $b_{j}\in B\backslash \{0\}$ is a period of the set $%
V+B=n_{2}p_{2}n_{3}\mathbb{I}_{n_{1}p_{1}}=\left\langle B\right\rangle :$%
\begin{eqnarray*}
V+U^{\prime }+B &=&V+\left[ \left( U\backslash \{u_{1},...,u_{n}\}\right)
\cup \{u_{1}+b_{1},...,u_{n}+b_{n}\}\right] +B \\
&=&V+\left( U+B)\backslash (\{u_{1},...,u_{n}\}+B\right) \cup
(\{u_{1}+b_{1},...,u_{n}+b_{n}\}+B) \\
&=&V+\left( U+B)\backslash (\{u_{1},...,u_{n}\}+B\right) \cup
(\{u_{1},...,u_{n}\}+B) \\
&=&V+U+B
\end{eqnarray*}%
and since $a_{i}\in A\backslash \{0\}$ is a period of $U+A=n_{1}p_{1}n_{3}%
\mathbb{I}_{n_{2}p_{2}}=\left\langle A\right\rangle :$%
\begin{eqnarray*}
U+V^{\prime }+A &=&U+\left[ \left( V\backslash \{v_{1},...,v_{m}\}\right)
\cup \{v_{1}+a_{1},...,v_{m}+a_{m}\}\right] +A \\
&=&U+\left( V+A)\backslash (\{v_{1},...,v_{m}\}+A\right) \cup
(\{v_{1}+a_{1},...,v_{m}+a_{m}\}+A) \\
&=&U+\left( V+A)\backslash \{v_{1},...,v_{m}\}+A\right) \cup
(\{v_{1},...,v_{m}\}+A) \\
&=&U+V+A
\end{eqnarray*}%
The sum $S+R$\ is then equal to%
\begin{eqnarray*}
S+R &=&\left( A+B\right) +\left( (U+V^{\prime }+K_{1})\cup (U^{\prime
}+V+K_{2})\right)  \\
&=&(A+B+U+V^{\prime }+K_{1})\cup (A+B+U^{\prime }+V+K_{2}) \\
&=&(A+B+U+V+K_{1})\cup (A+B+U+V+K_{2}) \\
&=&A+B+U+V+(K_{1}\cup K_{2}) \\
&=&A+B+U+V+K \\
&=&X+Y+K \\
&=&H+K=\mathbb{Z}_{N}
\end{eqnarray*}%
with $X=A+U$, $Y=B+V$ and $H=X+Y$. The sum is direct because the computation
of the cardinality leads to $\left\vert H\right\vert \left\vert K\right\vert
=N$.$\ $Furthermore, $S$ is non-periodic as in theorem \ref{thm_12} and if $R
$ is non-periodic, the pair $(S,R)$ is a Vuza canon.
\end{proof}

\begin{example}
Let $N=72$ with $n_{1}=p_{1}=n_{3}=2$, $n_{2}=p_{2}=3$. Compute the sets $%
A=\{0,8,16\}$, $B=\{0,18\}$, $U=\{0,24,48\}$ and $V=\{0,36\}$. The sets $%
U^{\prime }=\{0,24,48+18\}$ and $V^{\prime }=\{0,36+8\}$ are non-periodic.
The set $U+V^{\prime }=\{0,20,24,44,48,68\}$ has period 24 and 48.\ The set $%
U^{\prime }+V=\{0,24,30,36,60,66\}$ has period 36. Choosing $K_{2}=\{0\}$
and $K_{1}=\{1\}$ leads to a Vuza canon with%
\begin{equation*}
S=A\oplus B=\{0,8,16,18,26,34\}
\end{equation*}%
and%
\begin{eqnarray*}
R &=&(U+V^{\prime }+K_{1})\cup (U^{\prime }+V+K_{2}) \\
&=&\{0,1,21,24,25,30,36,45,49,60,66,69\}
\end{eqnarray*}
\end{example}

\begin{remark}
\label{rmk_01}Let $\mathcal{R}$ be the set of all possible sets $%
\{U+V^{\prime },U^{\prime }+V\},$ for all $U,U^{\prime },V$ and $V^{\prime }$%
. Suppose the sets $K_{1},K_{2},...,K_{n_{3}}$ exist, and let K be the union%
\begin{equation*}
K=\underset{j=1}{\overset{n_{3}}{\mathop{\displaystyle \bigcup }}}K_{j}
\end{equation*}%
If there exist $n_{3}$ different $R_{j}$ in $\mathcal{R}$, it is easy to
verify that the pair $(S,R)$ with $S$ as above and 
\begin{equation*}
R=\underset{j=1}{\overset{n_{3}}{\mathop{\displaystyle \bigcup }}}\left(
R_{j}+K_{j}\right)
\end{equation*}%
is a Vuza canon of $\mathbb{Z}_{N}$.
\end{remark}

Knowing a Vuza canon $(S,R)$, new canons can be generated by applying the
above construction to the $S$\ part. More precisely,

\begin{theorem}
\label{thm_S}Let $A,B,U,V,U^{\prime },V^{\prime },K_{1}$ and $K_{2}$ defined
as above and such that the pair $(S,R)$ with 
\begin{eqnarray*}
S &=&A+B \\
R &=&(U+V^{\prime }+K_{1})\cup (U^{\prime }+V+K_{2})
\end{eqnarray*}%
is a Vuza canon of $\mathbb{Z}_{N}$.\ Suppose that there is two non-periodic
sets $A^{\prime }$\ and $B^{\prime }$\ such that the set $A^{\prime }$ is
constructed by replacing non-zero elements $a_{j}\in A$ by their sum with
some non-zero elements of U 
\begin{equation*}
A^{\prime }=\left( A\backslash \{a_{1},...,a_{n}\}\right) \cup
\{u_{1}+a_{1},...,u_{n}+a_{n}\}
\end{equation*}%
and the set $B^{\prime }$ by replacing some non-zero elements $b_{i}\in B$
by their sum with some non-zero elements of V%
\begin{equation*}
B^{\prime }=\left( B\backslash \{b_{1},...,b_{m}\}\right) \cup
\{v_{1}+b_{1},...,v_{m}+b_{m}\}
\end{equation*}%
Then the pair $(S^{\prime },R)$ with $S^{\prime }=A^{\prime }+B^{\prime }$
is a Vuza canon of $\mathbb{Z}_{N}.$
\end{theorem}

\begin{proof}
The proof is straightforward since $A^{\prime }+U=A+U,$ and $B^{\prime
}+V=B+V.$ Notice that elements $u_{j}$ (resp.\ $v_{j}$) are not necessary
different.
\end{proof}

\begin{theorem}
\label{thm_L}Let $A,B,U,V,U^{\prime },V^{\prime },K_{1}$ and $K_{2}$ defined
as above and such that the pair $(S,R)$ with 
\begin{eqnarray*}
S &=&A+B \\
R &=&(U+V^{\prime }+K_{1})\cup (U^{\prime }+V+K_{2})
\end{eqnarray*}%
is a Vuza canon of $\mathbb{Z}_{N}$.\ Suppose that L and M are proper
subsets of $K=\mathbb{Z}_{n_{3}}$ such that $\mathbb{Z}_{n_{3}}=L\oplus M$
is a direct sum and $M=M_{1}\cup M_{2}$. The pair $(S^{\prime },R^{\prime })$
with 
\begin{eqnarray*}
S^{\prime } &=&A+B+L \\
R^{\prime } &=&(U+V^{\prime }+M_{1})\cup (U^{\prime }+V+M_{2})
\end{eqnarray*}%
is a rhythmic canon of $\mathbb{Z}_{N}$.
\end{theorem}

\begin{proof}
\begin{eqnarray*}
S^{\prime }+R^{\prime } &=&(A+B+L+U+V^{\prime }+M_{1})\cup (A+B+L+U^{\prime
}+V+M_{2}) \\
&=&A+B+L+U+V+(M_{1}\cup M_{2}) \\
&=&H+K=\mathbb{Z}_{N}
\end{eqnarray*}%
with $X=A+U$, $Y=B+V$ and $H=X+Y$. The sum is direct because the computation
of the cardinality leads to $\left\vert H\right\vert \left\vert K\right\vert
=N$.$\ $
\end{proof}

\section{Computing Vuza Canons}

From the musical point of view, canons could be translated without changing
their identity. If $(S,R)$ is a canon of $\mathbb{Z}_{N}$, any translate of $%
S$ tiles with the same $R$. That is why canons are identified by their \emph{%
prime form}. If $S$ is a set of $\mathbb{Z}_{N}$, the \emph{basic form} of $%
S $ is the smallest circular permutation for lexicographic order $\Delta S$\
of the set of consecutive intervals in $S:$ 
\begin{equation*}
(s_{2}-s_{1},s_{3}-s_{2},...,s_{n}-s_{n-1},s_{1}-s_{n})
\end{equation*}%
where $0\leq s_{1}<s_{2}<\cdots <s_{n}<N$ are the elements of $S$.\ For
example in $\mathbb{Z}_{72}$, the set $S=\{0,2,10,18,56,64\}$ has a prime
form equals to $\{0,8,16,18,26,34\}$ and a basic form equals to $%
(8,8,2,8,8,38)$.\ The results of the previous section provide us a way for
computing Vuza canons for a given Vuza order $N$.

(1) Determine a proper subset $H$ of $\mathbb{Z}_{N}$\ such that $H$\ is a
direct sum of two subsets $X$ and $Y$.\ For a given $%
N=n_{1}n_{2}n_{3}p_{1}p_{2},$ where $p_{1},p_{2}$\ are different primes, $%
p_{i}n_{i}\geq 2,$ for $i=1,2$ and $\mathrm{gcd}(n_{1}p_{1},n_{2}p_{2})=1$,
there always exists at least one such subset%
\begin{equation*}
H=n_{3}\mathbb{I}_{n_{1}p_{1}n_{2}p_{2}},\qquad X=n_{1}n_{3}p_{1}\mathbb{I}%
_{n_{2}p_{2}},\qquad Y=n_{2}n_{3}p_{2}\mathbb{I}_{n_{1}p_{1}}
\end{equation*}%
but it could exist more than one.\ For example, for $N=144$, we have two
solutions: $H=2\mathbb{I}_{72}=16\mathbb{I}_{9}+18\mathbb{I}_{8}$ and $H=4%
\mathbb{I}_{36}=36\mathbb{I}_{4}+16\mathbb{I}_{9}.$ Next, decompose $H$ as a
direct sum of $X$ and $Y$, and each set as $X=A+U$ and $Y=B+V$.

(2) Compute the list of all possible sets $A$ by theorem \ref{thm_S}, the
list of all possible sets $B$ and deduce all prime form of $S=A+B$.

(3) Compute the set of coset representatives such that $\mathbb{Z}%
_{N}=H\oplus K$.\ If $K$ can be decompose as a sum of two proper subsets $%
K=L+M$ as in theorem \ref{thm_L}, compute the new orbit of Vuza canons $%
S=A+B+L$ and if $S$ is non-periodic, memorize all prime forms.

(4) For a prime form $S$ of each orbit, obtained as above or by switching $%
n_{1}$ and $n_{3}$, compute the prime forms of $R$, using the results of the
previous section and remark \ref{rmk_01}.

The computation of the number of Vuza canons up to circular permutations
leads to the following tables.\ The length of $S$ is $\left\vert
S\right\vert =n_{1}n_{2}$ and the length of $R$ is $\left\vert R\right\vert
=n_{3}p_{1}p_{2}.$ The number of Vuza canons found is the product $\#S.\#R$
of the number $\#S$\ of prime forms $S$ times the number $\#R$\ of prime
forms $R$. The results of Fripertinger \citep{Fri2001}, Amiot \citep{Ami2011}
and Kolountzakis and Matolcsi \citep{Kol2009} depicted in the first table
are well-known.

\begin{table}[tbp]
\centering
\begin{tabular}{c|ccccc|c|c|c|c}
\hline
$N$ & $n_{1}$ & $p_{1}$ & $n_{2}$ & $p_{2}$ & $n_{3}$ & $\left\vert
S\right\vert $ & $\left\vert R\right\vert $ & $\#S$ & $\#R$ \\ \hline
$72$ & $2$ & $2$ & $3$ & $3$ & $2$ & $6$ & $12$ & $3$ & $6$ \\ 
$108$ & $2$ & $2$ & $3$ & $3$ & $3$ & $6$ & $18$ & $3$ & $252$ \\ 
$120$ & $2$ & $2$ & $3$ & $5$ & $2$ & $6$ & $20$ & $8$ & $18$ \\ 
$-$ & $2$ & $2$ & $5$ & $3$ & $2$ & $10$ & $12$ & $16$ & $20$ \\ 
$144$ & $2$ & $2$ & $3$ & $3$ & $4$ & $6$ & $24$ & $3$ & $8640$ \\ 
$-$ & $2$ & $2$ & $3$ & $3$ & $4$ & $6$ & $24$ & $6$ & $36$ \\ 
$-$ & $4$ & $2$ & $3$ & $3$ & $2$ & $12$ & $12$ & $6$ & $60$ \\ 
$-$ & $4$ & $2$ & $3$ & $3$ & $2$ & $12$ & $12$ & $162$ & $12$ \\ 
$-$ & $4$ & $2$ & $3$ & $3$ & $2$ & $12$ & $12$ & $324$ & $6$ \\ 
$168$ & $2$ & $2$ & $3$ & $7$ & $2$ & $6$ & $28$ & $16$ & $54$ \\ 
$-$ & $2$ & $2$ & $7$ & $3$ & $2$ & $14$ & $12$ & $104$ & $42$ \\ \hline
\end{tabular}%
\par
\bigskip
\caption{Classification of Vuza Canons (up to $N=168$)}
\end{table}

The second table shows new results obtained by application of the above
resullts. Due to too long computing times, some cases are missing (e.g.\ $%
N=216$ with $n_{1}=p_{1}=2$, $n_{2}=p_{2}=3$ and $n_{3}=6$).

\begin{table}[tbp]
\centering
\begin{tabular}{c|ccccc|c|c|c|c}
\hline
$N$ & $n_{1}$ & $p_{1}$ & $n_{2}$ & $p_{2}$ & $n_{3}$ & $\left\vert
S\right\vert $ & $\left\vert R\right\vert $ & $\#S$ & $\#R$ \\ \hline
$180$ & $2$ & $2$ & $3$ & $3$ & $5$ & $6$ & $30$ & $3$ & $77760$ \\ 
$-$ & $2$ & $2$ & $3$ & $5$ & $3$ & $6$ & $30$ & $8$ & $2052$ \\ 
$-$ & $2$ & $5$ & $3$ & $3$ & $2$ & $6$ & $30$ & $3$ & $84$ \\ 
$-$ & $2$ & $2$ & $5$ & $3$ & $3$ & $10$ & $18$ & $16$ & $1800$ \\ 
$-$ & $3$ & $3$ & $5$ & $2$ & $2$ & $15$ & $12$ & $9$ & $105$ \\ 
$200$ & $2$ & $2$ & $5$ & $5$ & $2$ & $10$ & $20$ & $125$ & $60$ \\ 
$216$ & $2$ & $2$ & $9$ & $3$ & $2$ & $18$ & $12$ & $575$ & $72$ \\ 
$-$ & $4$ & $2$ & $3$ & $3$ & $2$ & $12$ & $18$ & $6$ & $13680$ \\ 
$240$ & $4$ & $2$ & $5$ & $3$ & $2$ & $20$ & $12$ & $32$ & $200$ \\ 
$-$ & $4$ & $2$ & $3$ & $5$ & $2$ & $12$ & $20$ & $4100$ & $16$ \\ 
$252$ & $2$ & $2$ & $3$ & $7$ & $3$ & $6$ & $42$ & $16$ & $396$ \\ 
$-$ & $2$ & $7$ & $3$ & $3$ & $2$ & $6$ & $42$ & $9$ & $366$ \\ 
$264$ & $2$ & $2$ & $3$ & $11$ & $2$ & $6$ & $44$ & $40$ & $558$ \\ 
$270$ & $5$ & $2$ & $3$ & $3$ & $3$ & $15$ & $18$ & $9$ & $50400$ \\ 
$280$ & $2$ & $2$ & $5$ & $7$ & $2$ & $10$ & $28$ & $425$ & $180$ \\ 
$-$ & $2$ & $2$ & $7$ & $5$ & $2$ & $14$ & $20$ & $2232$ & $126$ \\ 
$300$ & $2$ & $3$ & $5$ & $5$ & $2$ & $10$ & $30$ & $104$ & $240$ \\ 
$-$ & $3$ & $2$ & $5$ & $5$ & $2$ & $15$ & $20$ & $104$ & $480$ \\ 
$324$ & $2$ & $2$ & $9$ & $3$ & $3$ & $18$ & $18$ & $729$ & $16848$ \\ 
$336$ & $4$ & $2$ & $3$ & $7$ & $2$ & $12$ & $28$ & $32$ & $7020$ \\ 
$-$ & $4$ & $2$ & $7$ & $3$ & $2$ & $28$ & $12$ & $208$ & $420$ \\ 
$392$ & $2$ & $2$ & $7$ & $7$ & $2$ & $14$ & $28$ & $16807$ & $378$ \\ 
$400$ & $4$ & $2$ & $5$ & $5$ & $2$ & $20$ & $20$ & $250$ & $2040$ \\ 
$450$ & $3$ & $3$ & $5$ & $5$ & $2$ & $15$ & $30$ & $375$ & $1920$ \\ \hline
\end{tabular}%
\par
\bigskip
\caption{Lower bounds for the number of some Vuza canons}
\end{table}

\section{Conclusion}

The aim of this paper was to show how to construct some Vuza canons and to
compute a minoration of their numbers for some value of $N$. We established
some theorems helping us to elaborate a new algorithm.\ But in many cases,
the number of solutions is quite large, and some shortcuts have to be found.

\section*{Acknowledgements}

The author would like to thank the anonymous reviewers for their valuable
comments and suggestions to improve the quality of the paper.

\bigskip


\end{document}